\documentclass[reqno,11pt]{amsart}
\usepackage{amsmath,amssymb,amsfonts,amsthm, amscd,indentfirst}
\usepackage{amsmath,latexsym,amssymb,amsmath,
	amscd,amsthm,amsxtra,mathrsfs}
 
\usepackage[top=3cm, bottom=3.5cm, right=2.5cm, left=2.5cm]{geometry}

\usepackage[hyphens]{url}
\usepackage{hyperref}
\usepackage{bookmark}
\usepackage{amsmath,thmtools,mathtools}
\allowdisplaybreaks 
\mathtoolsset{showonlyrefs=true}
\usepackage[msc-links]{amsrefs} 

\newcounter{mparcnt}

\usepackage{fancyhdr}
\usepackage{esint}
\usepackage{enumerate}
\usepackage{xcolor}

\usepackage{pictexwd,dcpic}
\usepackage{graphicx}
\usepackage{caption}

\usepackage{etoolbox}

\usepackage{graphicx}
\usepackage{caption}
\usepackage{slashed}

\usepackage{epsfig,here}
\usepackage{subfigure,here}

\newtheorem{theorem}{Theorem}[section]
\newtheorem{lemma}[theorem]{Lemma}
\newtheorem{proposition}[theorem]{Proposition}
\newtheorem{definition}[theorem]{Definition}
\newtheorem{remark}[theorem]{Remark}

\newcommand{\abs}[1]{\lvert#1\rvert}

\newcommand{\norm}[1]{\lVert#1\rVert}

\newcommand{\rd}{{\rm d}}
\newcommand{\rdV}{{\rm dV}}

\newcommand{\rVol}{{\rm Vol}}

\newcommand{\rid}{{\rm id}}

\newcommand{\D}{{\slashed{D}}}
\newcommand{\rRe}{{\rm Re}}

\newcommand{\p}{\partial}

\def\<{\langle}
\def\>{\rangle}
\def\S{\mathbb{S}}

\newcommand{\ra}{\rightarrow}

\newcommand{\eq}[1]{\begin{equation}\allowdisplaybreaks\begin{alignedat}{2} #1 \end{alignedat}\end{equation}}

\numberwithin{equation} {section}

\begin{document}

	
\title[A conformal lower bound of weighted Dirac eigenvalues on manifolds with boundary]{
A conformal lower bound of weighted Dirac eigenvalues on manifolds with boundary
}
\date{\today}


\author{Mingwei Zhang}
\address{ Wuhan University, School of Mathematics and Statistics, 430072 Wuhan, China; and 
Albert-Ludwigs-Universit\"at Freiburg, Mathematisches Institut, Ernst-Zermelo-Str. 1, D-79104 Freiburg, Germany}
\email{zhangmwmath@whu.edu.cn}

\begin{abstract}

For the weighted Dirac eigenproblem on a compact spin manifold with the chiral boundary condition
\eq{
    \begin{cases}
        \D\varphi = \lambda f\varphi \quad\hbox{in}\quad M,\\
        \mathbf{B} \varphi = 0 \quad\hbox{on}\quad \partial M,
    \end{cases}
}
we first give a lower bound of the eigenvalue using the relative Yamabe constant
\eq{
 \lambda^2 \geq \frac{n}{4(n-1)} Y(M,\p M,[g]),
}
then prove that equality holds if and only if (up to a conformal transformation) $M$ is a hemisphere and $\varphi$ is a Killing spinor. More generalizations are studied.


\end{abstract}


\maketitle

\section{Introduction}

A. Lichnerowicz \cite{Lichnerowicz63} used his formula (see \eqref{pointwise_S-L} below) to establish the fundamental result on closed Riemannian manifolds: 
\eq{
    \lambda_1(\D_g)^2 \geq \frac{1}{4}\inf_{M}R_g\,,
}
where $R_g$ is the scalar curvature and $\lambda_1(\D_g)$ is the smallest positive eigenvalue of the Dirac operator.
Later, T. Friedrich \cite{Friedrich80} improved the lower bound by proving the sharp inequality
\eq{
    \lambda_1(\D_g)^2 \geq \frac{n}{4(n-1)}\inf_{M}R_g,
}
with equality if and only if there exits a non-zero real Killing spinor.
Since then, there have been amount of results which improved the Friedrich inequality in different ways.

Since the Dirac operator is conformally invariant (observed by Hitchin \cite{Hitchin74}, see \eqref{conformal_Dirac}), it is nature to ask if its eigenvalues can be bounded from below by conformal invariants. In \cite{H86}, O. Hijazi proved for closed manifolds with dimension $n\geq3$ that
\eq{
    \lambda_1(\D_g)^2 \geq \frac{n}{4(n-1)}\lambda_1(L_g),
}
where $\lambda_1(L)$ is the smallest positive eigenvalue of the conformal Laplacian operator $L$ given by
\eq{
    L_g u \coloneqq -\frac{4(n-1)}{n-2}\Delta_g u+R_g u.
}
As a consequence, he obtained
\eq{
  \label{Hijazi_ineq}  \lambda_1(\D_g)^2\rVol(M,g)^{\frac{2}{n}} \geq \frac{n}{4(n-1)}Y(M,[g]),
}
where $Y(M,[g])$ is the (classical) Yamabe constant. Later, B\"ar \cite{Bar1992} proved the same inequality for case $n=2$, i.e. for a closed surface $(\Sigma,g)$,
\eq{
    \lambda_1(\D)^2 {\rm Area}(\Sigma,g) \geq 2\pi\chi(\Sigma).
}

Recently, \cite{WZ25_zero_mode} generalized the Hijazi and B\"ar inequality to an inequality for a conformal lower bound of weighted eigenvalues: on a closed manifold $(M^n,g)$ $(n\ge 2$)  with a non-zero function $f$, if there exists a spinor field $\varphi$ such that
\eq{
    \D\varphi = \lambda f\varphi,
}
then
\eq{\label{our_ineq}
    \lambda^2 \|f\|_{L^n(M)}^2\geq \frac{n}{4(n-1)}Y(M,[g]).
}
If $f$ is constant, then we recover the Hijazi and B\"ar inequality.
If $f$ is nowhere vanishing, then one can use a conformal transformation by considering $\tilde g=\abs{f}^2 g$ to reduce to the case that $f$ is constant. In general, this transformation does not work, since $f$ may admit zeros.
The same inequality holds true if we replace the weight $f$ by a general symmetric endomorphism $H$ on the spinor bundle $\Sigma M$, see \cite{WZ25_zero_mode}.
In particular, if the endomorphism $H$ is induced by a one-form $A$, i.e., $H\varphi=iA\cdot \varphi$, then \eqref{our_ineq} was obtained by Frank-Loss \cite{Frank_Loss_2024} and Reu{\ss} \cite{Reuss25} in their study of zero modes. 
Our proof in \cite{WZ25_zero_mode} is motivated by the work of \cite{H86}, \cite{Frank_Loss_2024} and \cite{Jurgen_Julio-Batalla}, especially the work of  Julio-Batalla \cite{Jurgen_Julio-Batalla}, where the author classified equality case in the Hijazi inequality, \eqref{Hijazi_ineq}. Equality in the work of Frank-Loss and Reu{\ss} was classified in \cite{WZ25_zero_mode}.  See also the work in \cite{BGH25}. 

The main objective of the paper is to generalize the results in  \cite{WZ25_zero_mode} to manifolds with boundary.
Let $(M,g)$ be an $n$-dimensional compact Riemannian manifold with boundary $\p M$. The analogy of classical Yamabe constant, the relative Yamabe constant $Y(M,\p M,[g])$, was introduced by J. Escobar \cites{Escobar90, Escobar92}, see Section \ref{sec2.6} for the definition. Based on his result, it is natural to ask if the relative Yamabe constant bounds the boundary Dirac eigenvalues from below. Here we have more than one choices for the boundary condition of a Dirac equation, including the chiral (CHI), the MIT bag, the J-boundary, the Atiyah-Patodi-Singer (APS) and its varieties. For discussions on various boundary conditions, we refer to \cites{Booss-Wof93, HMR02,Ginoux09}. In \cite{Raulot06}, S. Raulot proved for CHI and MIT bag cases that
\eq{\label{Raulot_ineq}
    \lambda_1(\D_g)^2 \geq \frac{n}{4(n-1)}\lambda_1(L_g),
}
and consequently
\eq{ \label{ineq_Y}
    \lambda_1(\D_g)^2\rVol(M,g)^{\frac{2}{n}} \geq \frac{n}{4(n-1)}Y(M,\p M,[g]).
}
He also classified the equality for CHI case, i.e. equality in \eqref{Raulot_ineq} holds if and only if $(M,g)=(\S^n_+,g_{{\rm st}})$.
Equality in \eqref{ineq_Y} is not addressed in \cite{Raulot06}.

In this paper, we classify equality, and
prove a more general result on a conformal lower bound of boundary weighted Dirac eigenvalues. 
On an $n$-dimensional compact Riemannian manifold $(M,g)$ with boundary $\p M$, we consider the following weighted eigenvalue problem with chiral boundary condition:
\eq{\label{intro_0-form_eq}
    \begin{cases}
        \D\varphi = \lambda f\varphi \quad\hbox{in}\quad M,\\
        \mathbf{B} \varphi = 0 \quad\hbox{on}\quad \partial M,
    \end{cases}
}
where $\lambda>0$ is a constant.  
Here the chiral boundary operator $\mathbf{B}$ is one of $\mathbf{B}^\pm$, see Section \ref{sec2.3}.

\begin{theorem}\label{main_thm}
Let $(M^n,g)$ $(n\geq2)$ be a compact spin manifold with boundary $\p M$ and $f$ a non-zero function. If $(M,g)$ admits a non-trivial solution $\varphi$ to \eqref{intro_0-form_eq}, where $\varphi\in L^p$ with $p>\frac{n}{n-1}$, then
\eq{\label{ineq_thm1.1}
 \lambda^2 \|f\|_{L^n (M)}^2 \geq \frac{n}{4(n-1)} Y(M,\p M,[g]).
}
Moreover, the following statements are equivalent:
\begin{enumerate}
    \item Equality in \eqref{ineq_thm1.1} holds true,
    \item  $\varphi$ is a $\pm\frac{1}{2}$-Killing spinor, up to a conformal transformation,
    \item  $(M,g)=(\S^n_+,g_{{\rm st}})$ and $\varphi$ is a $\pm\frac{1}{2}$-Killing spinor, up to a conformal transformation.
\end{enumerate}

\end{theorem}

If $f$ is constant, \eqref{ineq_thm1.1} was proved by  Raulot in \cite{Raulot06}.
The same inequality holds true for two generalized cases: the weight $f$ replaced by a general symmetric endomorphism $H$ on the spinor bundle $\Sigma M$, and the CHI boundary condition replaced by the MIT bag or J-boundary condition, see Section \ref{sec_generalization}.

The Dirac operator and spinor fields play an important role in both geometry and analysis. Since \cite{Hitchin74}, the study of conformal invariants related to the Dirac operator has attracted more and more interests of mathematicians. Here we list some of the related works. For spinorial conformal invariants, see \cites{A03, Raulot09, WZ25_zero_mode, Jurgen_Julio-Batalla}; for the related Dirac equations, see \cites{Borrelli21, Trebuchon25, Booss-Wof93, Qiu-Wu25, Ding-Li18}; for the related eigenvalue estimates, see \cites{Ji-Lin21, WZ25_zero_mode, ABLO21, Raulot06, Reuss25, Frank_Loss_2024}; for other related works on the Dirac operator, see \cites{CJLW06, WZ25, JWZ07, FL1}.

\ 

\noindent{\it Organization of the rest of the paper:} In Section 2 we review briefly spin geometry, the Dirac operator, the relative Yamabe constant, and the boundary Dirac equation together with its conformal invariance. In Section 3, we introduce a family of conformal invariants, and relate it to the relative Yamabe constant. In Section 4, we complete the proof of Theorem \ref{main_thm}. In Section 5, we discuss two kinds of generalizations, one of the weight $f$ and the other of the boundary operator $\mathbf{B}$. In Section 6, we given an application to the energy gap of a ground state solution. The proof for the regularity is given in Appendix A for the completeness.

\section{Preliminaries}

\subsection{Spin manifolds and the Dirac operator}

In this subsection we recall some basics about spin manifolds and the Dirac operator. For details, we refer to \cites{Lawson, Baum_Book_90, Friedrich_Book, Ginoux09}.

Let $(M,g)$ be a Riemannian manifold of dimension $n\geq 2$. Assume $M$ is orientable, then it admits an ${\rm SO}(n)$-principle bundle $P_{{\rm SO}(n)}M$. Each fibre is an oriented orthonormal basis. If $P_{{\rm SO}(n)}M$ can be lifted up to $P_{{\rm Spin}(n)}M$, then we say $M$ is a spin manifold. The Lie group ${\rm Spin}(n)$ is the universal cover of ${\rm SO}(n)$, which is 2-fold. We call the quotient map $\sigma: P_{{\rm Spin}(n)}M \ra P_{{\rm SO}(n)}M$ a spin structure. In this paper, each manifold is supposed to be spin. It is well known that $M$ is spin if and only if the second Stiefel-Whitney class of $M$ vanishes. In particular, $\S^n$ is a spin manifold.

Let $\Sigma M$ be the associated complex vector bundle of $P_{{\rm Spin}(n)}M$, then $\Sigma M$ has complex rank $2^{[\frac n2]}$. The Riemannian metric $g$ on $M$ endows a canonical Hermitian metric $\<\cdot,\cdot\>$ and the corresponding spin connection $\nabla$ on $\Sigma M$.

A \textit{spinor field} is defined to be a section of $\Sigma M$. Tangent vector fields act on spinor fields by $\gamma: TM \ra {\rm End}_{\mathbb{C}}(\Sigma M)$, denoted by $X\cdot\psi\coloneqq \gamma(X)(\psi)$, which is called the Clifford multiplication. It satisfies $X\cdot Y\cdot \psi + Y\cdot X\cdot \psi = -2 g(X,Y)\psi$, and is anti-symmetric w.r.t. $\<\cdot,\cdot\>$.

Let $\{e_{j}\}_{j=1}^{n}$ be an orthonormal frame of $M$. The Dirac operator $\D:\Gamma(\Sigma M) \ra \Gamma(\Sigma M)$ is defined by
\eq{
    \D\psi \coloneqq \sum_{j=1}^{n} e_{j}\cdot \nabla_{e_{j}}\psi, \quad\forall \,\psi\in\Gamma(\Sigma M).
}
It is well known that $\D$ is a first-order elliptic operator, which plays the role as the square root of Laplacian-Beltrami via the Schr\"odinger-Lichnerowicz formula 
\eq{\label{pointwise_S-L}
    \D^2 = -\Delta + \frac{R}{4},
}
where $R$ is the scalar curvature. A spinor field $\psi$ is called a \textit{Killing spinors}, if
\eq{
    \nabla_{X}\psi = b X\cdot \psi, \quad \forall \,X\in\Gamma(TM),
}
where the constant $b\in\mathbb{C}$ is called the Killing number.
In particular, $\psi$ is called a real Killing spinor if $b\in\mathbb{R}$.
Existence of a non-trivial Killing spinor is a highly demanding requirement of the manifold.
The classification of complete simply connected Riemannian spin manifolds that carry a non-zero space of $b$-Killing spinor was given by M. Wang \cite{McWang1989} for $b=0$, C. B\"ar \cite{Bar1993} for $b\in\mathbb{R}\backslash\{0\}$ and H. Baum, etc. \cite{Baum_Book_90} for $b\in i\mathbb{R}\backslash\{0\}$.
In this paper, all Killing spinors are supposed to be real, therefore without loss of generality we may assume $b=\pm \frac 12 $.

\subsection{The boundary Dirac operator}

Let $(M,g)$ be an $n$-dimensional compact Riemannian manifold with non-empty boundary $\p M$ which is an oriented hypersurface of $M$ with the induced orientation. We denote by $\nu$ the unit inner normal vector field on $\p M$. It is well known that $\nu$ induces a spin structure on $\p M$ by restriction, denoted by $\Sigma\p M$.
Note that $\Sigma\p M$ is slightly different from the intrinsic spinor bundle $\Sigma_{{\rm int}}\p M$ on $\p M$. In fact, according to the chirality of spin representation, $\Sigma\p M$ coincides with $\Sigma_{{\rm int}}\p M$ if $n$ is odd, and with $\Sigma_{{\rm int}}\p M\oplus \Sigma_{{\rm int}}\p M$ if $n$ is even.
For the details we refer to \cites{Bar_Note, Ginoux09, Trautman95, Raulot06}.

The associated Clifford multiplication on $\p M$ is
\eq{
    X\cdot^{\p M}\psi \coloneqq X\cdot\nu\cdot\psi, \quad\forall X\in\Gamma(T\p M).
}
We denote by $\nabla,\nabla^{\p M}$ the Levi-Civita connection on $M,\p M$, respectively.
Recall the Gau\ss{} equation
\eq{
    \nabla_X Y = \nabla^{\p M}_X Y + g(A(X), Y)\nu, \quad\forall X,Y\in\Gamma(T\p M),
}
where $A(X)\coloneqq -\nabla_X\nu$ is the shape operator of $\p M$. 
The associated spin connection satisfies the spinorial Gau\ss{} equation on $\p M$
\eq{
    \nabla_X\psi = \nabla^{\p M}_X\psi + \frac{1}{2} A(X)\cdot^{\p M} \psi = \nabla^{\p M}_X\psi + \frac{1}{2} A(X)\cdot \nu \cdot \psi, \quad\forall X\in\Gamma(T\p M).
}
Let $\{e_1,\cdots,e_{n-1},e_n=\nu\}$ be an orthonormal frame of $TM$.
The associated Dirac operator on $\p M$, the boundary Dirac operator, is defined by
\eq{\label{boundary_Dirac}
    \D^{\p M}\psi \coloneqq \sum_{j=1}^{n-1} e_j\cdot^{\p M}\nabla^{\p M}_{e_j}\psi = -\nu\cdot\sum_{j=1}^{n-1}e_j\cdot\nabla_{e_j}\psi + \frac{n-1}{2}H\psi = -\nu\cdot\D\psi - \nabla_{\nu}\psi + \frac{n-1}{2}H\psi,
}
where $H\coloneqq\frac{1}{n-1}{\rm tr}(A)$ is the mean curvature of $\p M$. 
Note that $\D^{\p M}$ is slightly different from the intrinsic Dirac operator $\D^{\p M}_{{\rm int}}$ on $\p M$. In fact, according to the chirality of spin representation, $\D^{\p M}$ coincides with $\D^{\p M}_{{\rm int}}$ if $n$ is odd, and with $\D^{\p M}_{{\rm int}}\oplus -\D^{\p M}_{{\rm int}}$ if $n$ is even. 
A direct consequence is
\eq{\label{anticommute_D_nu}
    \D^{\p M}(\nu\cdot\psi) = - \nu\cdot \D^{\p M}\psi.
}
Therefore, ${\rm Spec}(\D^{\p M})$ is symmetric w.r.t. $0$.

\subsection{The chiral boundary operator}\label{sec2.3}

On an $n$-dimensional compact Riemannian manifold $(M,g)$ with boundary $\p M$, the chiral boundary operator $\mathbf{B}^\pm:\Gamma(\Sigma\p M,g)\to\Gamma(\Sigma\p M,g)$ is defined by
\eq{
    \mathbf{B}^\pm\varphi \coloneqq \frac{1}{2}(\rid \pm \nu\cdot G)\varphi,
}
acting on the restriction of the spinor bundle to the boundary, where $G:\Gamma(\Sigma M,g)\to \Gamma(\Sigma M,g)$ is an endomorphism which satisfies
\eq{\label{def_G}
    G^2 = \rid, \quad G^* = G, \quad \nabla G = 0, \quad X\cdot G\varphi = - G(X\cdot\varphi), \quad\forall X\in\Gamma(TM).
}
A direct consequence is
\eq{\label{commute_D_G}
    G\D^{\partial M} = \D^{\partial M}G.
}

We remark that if $n$ is even, then
\eq{
    G=\begin{cases}
        \omega\cdot \quad\hbox{if}\ n\equiv0\!\!\mod 4\\
        i\omega\cdot \quad\hbox{if}\ n\equiv2\!\!\mod 4
    \end{cases}
}
is a canonical choice of $G$, where $\omega\,\cdot:\Gamma(\Sigma M,g)\to \Gamma(\Sigma M,g)$ is the Clifford multiplication given by volume element, i.e. $\omega\,\cdot \coloneqq e_1\cdot\cdots\cdot e_n\cdot$.

\subsection{The weighted eigenvalue problem}

On an $n$-dimensional compact Riemannian manifold $(M,g)$ with boundary $\p M$, we consider the following weighted eigenvalue problem with chiral boundary condition:
\eq{\label{0-form_eq}
    \begin{cases}
        \D\varphi = \lambda f\varphi \quad\hbox{in}\quad M,\\
        \mathbf{B} \varphi = 0 \quad\hbox{on}\quad \partial M,
    \end{cases}
}
where $\lambda>0$ and $\norm{f}_n=1$. Here the chiral boundary operator $\mathbf{B}$ is one of $\mathbf{B}^\pm$.

We recall important conformal properties of the Dirac operator.
Let $\tilde{g}=h^{\frac{4}{n-2}}g$. The isometry $(TM,g) \to (TM, \tilde{g})$ with $X\mapsto \tilde{X}=h^{-\frac{2}{n-2}}X$ induces a canonical isometry $F: (\Sigma M,g) \to (\Sigma M, \tilde{g})$ with $\varphi \mapsto \tilde{\varphi}$. It preserves the Hermitian metric on spinor bundle, and the induced spin representation $\tilde{\cdot}$ satisfies $\widetilde{X\cdot\varphi} = \tilde{X}\,\tilde{\cdot}\,\tilde{\varphi}$. The induced Dirac operator $\D_{\tilde{g}}$ satisfies the well-known property (see for instance \cite{Ginoux09}*{Proposition 1.3.10})
\eq{\label{conformal_Dirac}
     \D_{\tilde{g}}(h^{-\frac{n-1}{n-2}}\tilde{\varphi}) = h^{-\frac{n+1}{n-2}}\widetilde{\D_g\varphi},
}
and similarly for the boundary Dirac operator (see for instance \cite{HMZ02})
\eq{\label{conformal_boundary_Dirac}
    \D^{\p M}_{\tilde{g}}(h^{-1}\tilde{\varphi}) = h^{-\frac{n}{n-2}}\widetilde{\D^{\p M}_g\varphi}.
}
Moreover, the chiral boundary operator is conformally invariant in the sense that (see also \cite{Raulot09})
\eq{
    \tilde{G} \coloneqq F \circ G \circ F^{-1} : \Gamma(\Sigma M,\tilde{g})\to \Gamma(\Sigma M,\tilde{g})
}
is a chiral boundary operator w.r.t. $\tilde{g}$.

\begin{definition}\label{Def2}
The conformal class of $(g, f)$ is defined by
\eq{
[g,f] \coloneqq \{(\tilde g, f_{\tilde g}) \,:\, \tilde g=h^{\frac{4}{n-2}}g,\ f_{\tilde g}=h^{-\frac{2}{n-2}}f \ \hbox{for some positive function}\ h\}.
}
\end{definition}

\begin{lemma}[Conformal invariance]\label{conformal_invariance}
Equation \eqref{0-form_eq} is conformally invariant in the following sense: Let $(\varphi,f)$ be a solution to \eqref{0-form_eq} w.r.t. metric $g$, and $\tilde g =h^{\frac 4{n-2}} g$.
Define 
\eq{\label{conformal_f}
     \varphi_{\tilde{g}} \coloneqq h^{-\frac{n-1}{n-2}}\tilde{\varphi}, \qquad f_{\tilde{g}} \coloneqq h^{-\frac{2}{n-2}}f.
}
Then $(\varphi_{\tilde g},f_{\tilde g})$ is a solution to \eqref{0-form_eq} w.r.t. metric $\tilde{g}$ with the same $\lambda$.
\end{lemma} 
     
\begin{proof}
In fact, \eqref{conformal_Dirac} and \eqref{0-form_eq} imply
\eq{
    \D_{\tilde{g}}\varphi_{\tilde{g}} = h^{-\frac{n+1}{n-2}}\widetilde{\D_g\varphi} = \lambda h^{-\frac{n+1}{n-2}}f\tilde{\varphi} = \lambda h^{-\frac{n+1}{n-2}}\cdot h^{\frac{2}{n-2}}f_{\tilde{g}}\cdot h^{\frac{n-1}{n-2}}\varphi_{\tilde{g}} = \lambda f_{\tilde{g}}\varphi_{\tilde{g}},
}
where
\eq{
    \int \abs{f_{\tilde{g}}}^n  \rd V_{\tilde{g}} = \int \abs{h^{-\frac{2}{n-2}} f}^n \cdot h^{\frac{2n}{n-2}}\, \rd V_g = \int \abs{f}^n \rd V_g = 1.
}
Moreover, since $\mathbf{B}\varphi=0$ on $\p M$, we have
\eq{
    \tilde{\mathbf{B}}\varphi_{\tilde{g}} = \frac{1}{2}(\rid \pm \tilde{\nu}\,\tilde{\cdot}\, \tilde{G})\varphi_{\tilde{g}} = h^{-\frac{n-1}{n-2}} \frac{1}{2} ( \tilde{\varphi} \pm \widetilde{ \nu\cdot G\varphi } ) = h^{-\frac{n-1}{n-2}} \widetilde{\mathbf{B}\varphi} = 0 \quad\hbox{on}\ \p M.
}

\end{proof}
    
Therefore, in order to estimate $\lambda$, one can choice a suitable conformal metric.

\subsection{The integral Schr\"odinger-Lichnerowicz formula}

On an $n$-dimensional compact Riemannian manifold $(M,g)$ with boundary $\p M$, the Dirac operator $\D$ is in general not $L^2$-self-adjoint. In fact, it is easy to check for any spinor fields $\psi$ and $\varphi$ that
\eq{\label{general_not_self_adjoint}
    \int_M \<\D\psi,\varphi\> - \int_M \<\psi,\D\varphi\> = -\int_{\p M}\<\nu\cdot\psi,\varphi\>.
}
We recall the well known identity
\eq{\label{well_known_identity}
    \abs{\nabla\varphi}^2 = \frac{1}{n}\abs{\D\varphi}^2 + \abs{P\varphi}^2,
}
where $P$ is the twistor operator (also called the Penrose operator) defined by $P_X\varphi \coloneqq \nabla_X\varphi + \frac{1}{n}X\cdot\varphi$.
Using \eqref{pointwise_S-L}, \eqref{boundary_Dirac}, \eqref{general_not_self_adjoint} and \eqref{well_known_identity} we obtain the integral Schr\"odinger-Lichnerowicz formula
\eq{\label{integral_S-L}
    \int_{\partial M} \Big( \<\D^{\partial M}\varphi,\varphi\> - \frac{n-1}{2}H\abs{\varphi}^2 \Big) = \int_M \Big( \frac{R}{4}\abs{\varphi}^2 - \frac{n-1}{n}\abs{\D\varphi}^2 \Big) + \int_M \abs{P\varphi}^2.
}

In particular, let us assume the chiral boundary condition. In this case, $\D$ is $L^2$-self-adjoint. Namely, if
\eq{
    \mathbf{B}\psi = \mathbf{B}\varphi = 0 \quad\hbox{on}\ \p M,
}
then for $\mathbf{B}=\mathbf{B}^\pm$ we have
\eq{
    G\psi = \pm\nu\cdot\psi, \quad G\varphi = \pm\nu\cdot\varphi \quad\hbox{on}\ \p M,
}
hence
\eq{\label{vanishing_boundary_nu_cdot}
    \<\psi,\nu\cdot\varphi\> 
    = \pm\<\psi,G\varphi\>
    = \pm\<G\psi,\varphi\>
    = \<\nu\cdot\psi,\varphi\>
    = -\<\psi,\nu\cdot\varphi\> \quad\hbox{on}\ \p M.
}
Therefore $\<\psi,\nu\cdot\varphi\> = 0$ on $\p M$, and
\eq{
    \int_M\<\D\psi,\varphi\> = \int_M\<\psi,\D\varphi\>.
}
Moreover, if $\mathbf{B}\varphi = 0$ on $\p M$, then we have
\eq{
    \<\D^{\p M}\varphi,\varphi\> = \<\nu\cdot G\D^{\p M}\varphi, \nu\cdot G\varphi\> = -\<\D^{\p M}(\nu\cdot G\varphi),\nu\cdot G\varphi\> = -\<\D^{\p M}\varphi,\varphi\>,
}
where we have used \eqref{anticommute_D_nu} and \eqref{commute_D_G}. Hence
\eq{\label{vanishing_boundary_Dirac}
    \<\D^{\p M}\varphi,\varphi\> = 0 \quad\hbox{on}\ \p M.
}

We remark that since the chiral boundary condition is local elliptic and makes $\D$ $L^2$-self-adjoint, the following eigenproblem
\eq{\label{chiral_eigenspinor}
    \begin{cases}
        \D\varphi = \lambda \varphi \quad\hbox{in}\quad M,\\
        \mathbf{B} \varphi = 0 \quad\hbox{on}\quad \partial M,
    \end{cases}
}
has a real and discrete spectrum, see also \cite{HMR02}. One can also see from the estimate given in \cite{CJSZ2018} that $0\not\in{\rm Spec}$. The lower bound of $\lambda$ given by the relative Yamabe constant was proved in \cite{Raulot06}, which is a direct consequence of Theorem \ref{main_thm}.

\subsection{The relative Yamabe constant}\label{sec2.6}

The classical Yamabe constant is defined on a closed manifold by 
\eq{\label{classical_Yamabe_constant}
    Y(M,[g]) \coloneqq \inf _{0\not\equiv u\in W^{1,2} } \frac {\int_M uL_g u \,\rdV_g}{(\int _M u^{\frac {2n}{n-2}}\rdV_g)^{\frac{n-2}{n}}},
}
where $L_g$ is the conformal Laplacian given by
\eq{
    L_g u \coloneqq -\frac{4(n-1)}{n-2}\Delta_g u+R_g u.
}
A conformal metric $\tilde g= h^{\frac {4}{n-2}}g$ has scalar curvature
\eq{\label{scalar_curvature}
    R_{\tilde g}=h^{-\frac {n+2}{n-2}} L_g h  = h^{-\frac {n+2}{n-2}} \left\{-\frac{4(n-1)}{n-2}\Delta_g h+R_g h\right\}.
}
For a compact manifold with boundary, the analogy is defined by (see \cite{Escobar92})
\eq{\label{def_relative_Yamabe_constant}
    Y(M,\p M,[g]) \coloneqq \inf _{0\not\equiv u\in W^{1,2} } \frac {\int_M \big( \frac{4(n-1)}{n-2}\abs{\nabla u}_g^2 + R_g u^2 \big) \,\rdV_g + 2(n-1)\int_{\p M}Hu^2 \rd s_g }{(\int _M u^{\frac {2n}{n-2}}\rdV_g)^{\frac{n-2}{n}}}.
}
A conformal metric $\tilde g= h^{\frac {4}{n-2}}g$ has mean curvature of $\p M$
\eq{\label{mean_curvature}
    H_{\tilde g} = h^{-\frac{n}{n-2}} B_g h  = h^{-\frac{n}{n-2}} \left\{\frac{2}{n-2}\frac{\p h}{\p \nu}+H_g h\right\},
}
where $B_g$ is the conformal boundary operator given by
\eq{
    B_g u \coloneqq \frac{2}{n-2}\frac{\p u}{\p \nu}+H_g u.
}
The operators $L$ and $B$ transform nicely under the conformal change:
\eq{\label{conformal_L_and_B}
    L_{\tilde{g}} = h^{-\frac{n+2}{n-2}}L_g(hu), \qquad B_{\tilde{g}} = h^{-\frac{n}{n-2}}B_g(hu).
}

The classical Yamabe problem has a satisfying resolution, due to Yamabe, Trudinger, Aubin and Schoen, that the infimum in \eqref{classical_Yamabe_constant} is achieved by a conformal metric $\tilde{g}$, which is called a Yamabe minimizer. For an introduction of the Yamabe problem, we refer to \cite{Lee_Parker}. If $(M,g)$ is Einstein, then Obata \cite{Obata71} implies that any other conformal metric $\tilde{g}$ with constant scalar curvature is also Einstein. If in addition $(M,g)$ is not conformally equivalent to $(\S^n,g_{{\rm st}})$, then $\tilde g=cg$ for some constant $c>0$, hence any conformal metric of constant scalar curvature is a Yamabe minimizer.

For the relative Yamabe problem, Escobar \cite{Escobar92} proved that the counterpart of the Yamabe minimizer has constant scalar curvature and zero mean curvature on boundary. The counterpart of Obata's rigidity was proved by \cite{Reilly77} for Dirichlet case and by \cite{Escobar90} for the Neumann case.
We note that that the Yamabe constant of the sphere $\S^n$ is
\eq{
    Y(\S^n) = \frac{4(n-1)}{n-2} S_n = n(n-1)\omega_n^{\frac{2}{n}},
}
and the relative Yamabe constant of the hemisphere $\S^n_+$ is
\eq{\label{relative_Yamabe_constant_hemisphere}
    Y(\S^n_+,\p\S^n_+) = n(n-1)\Big(\frac{\omega_n}{2}\Big)^{\frac{2}{n}}.
}


\section{A family of conformal invariants}\label{sec_a_family_of_conformal_invariants}

Let us introduce a family of eigenvalue problems, from which we then define a family of conformal invariants. It is interesting to observe that these invariants turn out to be identical, see Proposition \ref{prop3.5} below. For the case of a closed manifold, the result is similar, see \cite{WZ25_zero_mode}.

We define for $a\in (0,1]$
\eq{
    L_g^a \coloneqq -a\frac{4(n-1)}{n-2}\Delta_g + R_g, \qquad B_g^a \coloneqq \frac{2a}{n-2}\frac{\partial}{\partial \nu} + H.
}
and consider the following weighted Robin-eigenproblem:
\eq{
    \begin{cases}
        L_g^a u = \mu_a(g,f)f^2u \quad\hbox{in}\ M,\\
        B_g^a u = 0 \quad\hbox{on}\ \partial M.
    \end{cases}
}
The boundary condition implies
\eq{
    \int_M u(-\Delta)u = \int_M \abs{\nabla u}^2 - \int_{\partial M} u\frac{\p u}{\p \nu} = \int_M \abs{\nabla u}^2 + \frac{n-2}{2a}\int_{\partial M}Hu^2.
}
The first eigenvalue is given by
\eq{\label{def_mu_a}
    \mu_a(g,f) &= \inf_{u\not\equiv0} \frac{\int_M uL_g^au \,\rd V_g + 2(n-1)\int_{\partial M} uB_g^au \,\rd s_g}{\int_M f^2 u^2 \,\rd V_g}\\
    &= \inf_{u\not\equiv0} \frac{\int_M \big( a \frac {4(n-1)}{n-2}\abs{\nabla u}_g^2 + R_g u^2 \big) \,\rd V_g + 2(n-1)\int_{\partial M}Hu^2 \rd s_g}{\int_M f^2 u^2 \,\rd V_g}.
}

\begin{definition}\label{def_gamma_a}
For any given $a\in (0, 1]$, we define a conformal invariant for the class $[g,f]$
\eq{
    \gamma_a([g,f]) \coloneqq \sup_{(\tilde g, f_{\tilde g})\in [g,f]} \mu_a (\tilde g, f_{\tilde g}).
}
    
\end{definition}

We will prove that all $\gamma_a$'s are actually the same.
First of all, we observe that the case $a=1$ is special, since $L^{1}_g = L_g$ is the conformal Laplacian and $B^{1}_g = B_g$ is the conformal boundary operator.

\begin{lemma}\label{lem_conformal_invariance_mu_1}
$\mu_1(g,f)$ is conformally invariant, namely $\mu_1(g,f)=\mu_1(\tilde{g},f_{\tilde{g}})$ for any $\tilde{g}=h^{\frac{4}{n-2}}g$.
\end{lemma}

\begin{proof}
Set
\eq{
    I_{g,f}(u) \coloneqq \frac{\int_M uL_gu \,\rd V_g + 2(n-1)\int_{\partial M} uB_gu \,\rd s_g}{\int_M f^2 u^2 \,\rd V_g}.
}
Using \eqref{conformal_L_and_B} we have
\eq{
    \int_M u L_{\tilde{g}}u \,\rd V_{\tilde{g}} &= \int_M h^{-1}(hu) \cdot h^{-\frac{n+2}{n-2}} L_{g}(hu)\cdot h^{\frac{2n}{n-2}} \,\rd V_{g} = \int_M (hu) L_g(hu) \,\rdV_g,\\
    \int_{\p M} u B_{\tilde{g}}u \,\rd s_{\tilde{g}} &= \int_{\p M} h^{-1}(hu) \cdot h^{-\frac{n}{n-2}} B_{g}(hu)\cdot h^{\frac{2n-2}{n-2}} \,\rd s_{g} = \int_{\p M} (hu) B_g(hu) \,\rd s_g
}
and
\eq{
    \int_M f_{\tilde{g}}^2u^2 \rd V_{\tilde{g}} &= \int_M h^{-\frac{4}{n-2}}f^2\cdot h^{-2}(hu)^2\cdot h^{\frac{2n}{n-2}} \rdV_g = \int_M f^2(hu)^2 \rdV_g.
}
Hence
\eq{\label{eq_lem3.2}
    I_{\tilde{g},f_{\tilde{g}}}(u) = I_{g,f}(hu).
}
Since
\eq{
    \mu_1(g,f) = \inf_{u\not\equiv0} I_{g,f}(u),
}
we complete the proof.

\end{proof}

In contrast, $\mu_a (g,f)$ $(0<a<1)$ is not conformally invariant. Nevertheless, its supreme over $[g,f]$, $\gamma_a([g,f])$, has the following property.

\begin{proposition}\label{gamma_a_all_same}
For any $a\in (0, 1]$,
\eq{
    \gamma_a([g,f]) = \gamma_1([g,f]).
}
Moreover, $\gamma_a([g,f])$ is achieved by a conformal metric $\tilde{g} \in [g]$.
\end{proposition}

\begin{proof}
The proof is the same as that of \cite{WZ25_zero_mode}*{Proposition 3.5}. Here in this case, $\gamma_a([g,f])$ is achieved by $\tilde{g}=u_0^{\frac{4}{n-2}}g$, where the positive function $u_0$ is a weighted eigenfunction of $L_g$ given by
\eq{\label{weighted_first_eigenfunction}
    \begin{cases}
        L_g u_0 = \gamma_1([g,f])f^2u_0 \quad\hbox{in}\ M,\\
        B_g u_0 = 0 \quad\hbox{on}\ \partial M,
    \end{cases}
}
and $\tilde{g}$ has curvature
\eq{\label{constant_curvature}
    R_{\tilde{g}} = \gamma_1([g,f]) f_{\tilde{g}}^2, \quad H_{\tilde{g}} = 0.
}
  
\end{proof}

Finally we relate it to the relative Yamabe constant.

\begin{proposition}\label{prop3.5} 
For any $a \in (0, 1]$,
    \eq{\gamma_a ([g,f])=\gamma_1([g,f])  \geq Y(M,\p M,[g]).}
\end{proposition}

\begin{proof}
Let $u_0$ be the positive function achieving the infimum in \eqref{def_gamma_a} for $a=1$.
By H\"older's inequality we have
\eq{
    \int_M f^2u_0^2 \rdV_g \leq \Big( \int_M u_0^{\frac{2n}{n-2}} \rdV_g \Big)^{\frac{n-2}{n}} \Big( \int_M \abs{f}^n \rdV_g \Big)^{\frac{2}{n}} = \Big( \int_M u_0^{\frac{2n}{n-2}} \rdV_g \Big)^{\frac{n-2}{n}},
}
where we have used the normalization $\norm{f}_{L^n}=1$. In view of \eqref{def_relative_Yamabe_constant}, we complete the proof.

\end{proof}

\begin{theorem}\label{prop_gamma_a}
Let $(M^n,g)$ $(n\geq3)$ be  a compact spin manifold with boundary $\p M$. If $(M,g)$ admits a non-trivial solution $(\varphi,f)$ to \eqref{0-form_eq}, where $\varphi\in L^p$ with $p>\frac{n}{n-1}$ and $f\in L^n$, then
\eq{\label{ineq_thm3.6}
    \lambda^2 \geq \frac{n}{4(n-1)}\gamma_1([g,f]).
}
\end{theorem}

\begin{proof}
First of all, Lemma \ref{lem_regularity_L^r} implies that $\varphi\in L^{\frac {2n}{n-2}}$, which in turn implies that $A\cdot \varphi \in L^2 $ and $\varphi \in W^{1,2}$.

We recall the integral Schr\"odinger-Lichnerowicz formula \eqref{integral_S-L}
\eq{
    \int_{\partial M} \Big( \<\D^{\partial M}\varphi,\varphi\> - \frac{n-1}{2}H\abs{\varphi}^2 \Big) = \int_M \Big( \frac{R}{4}\abs{\varphi}^2 - \frac{n-1}{n}\abs{\D\varphi}^2 \Big) + \int_M \abs{P\varphi}^2.
}
It is easy to show that \eqref{integral_S-L} is valid not only for smooth $\varphi$, but also for any $\varphi\in W^{1,2}$.

We now prove \eqref{ineq_thm3.6} by contradiction. Assume that $\lambda^2 \leq \frac{n}{4(n-1)}(\gamma_1([g,f]) - \delta)$
for some $\delta>0$. For any $a\in (0,1]$, by Proposition \ref{gamma_a_all_same}, $\tilde{g} = u_0^{\frac{4}{n-2}}g$ achieves the supreme in Definition \ref{def_gamma_a}, where $u_0$ is the weighted first eigenfunction given by \eqref{weighted_first_eigenfunction}. Hence by \eqref{constant_curvature} we have $R_{\tilde{g}}>0$, $H_{\tilde{g}}=0$, and
\eq{
    \mu_a(\tilde{g},f_{\tilde{g}}) = \gamma_a([g,f]) = \gamma_1([g,f]) \eqcolon Y.
}
Inserting \eqref{0-form_eq} \eqref{vanishing_boundary_Dirac} and \eqref{constant_curvature} into the integral Schr\"odinger-Lichnerowicz formula \eqref{integral_S-L} yields
\eq{\label{crucial_integral_identity}
    0 = \int_M \big( R_{\tilde{g}} - \frac{4(n-1)}{n}\lambda^2f^2 \big)\abs{\varphi_{\tilde{g}}}_{\tilde{g}}^2 +4 \int \abs{P_{\tilde{g}}\varphi_{\tilde{g}}}_{\tilde{g}}^2.
}
Now the conclusion follows from the same argument in the proof of \cite{WZ25_zero_mode}*{Theorem 3.8}.

\end{proof}

\section{Proof of Theorem \ref{main_thm}}\label{sec_proof_thm1.1}

In this section we prove our first main result.
We start with the following lemma. When $(M,g)=(\S^n,g_{{\rm st}})$, it was proved in \cite{WZ25}.

\begin{lemma}\label{lem4.1}
Let $(M,g)$ be a compact spin manifold with or without boundary. If $\varphi$ is a $\pm\frac{1}{2}$-Killing spinor and $\psi$ is a $\lambda$-eigenspinor of $\D$, then $\<\varphi,\psi\>$ is a $(\lambda^2\pm\lambda-\frac{n^2-2n}{4})$-eigenfunction of $-\Delta$.
\end{lemma}

\begin{proof}
Since $\varphi$ is a $\pm\frac{1}{2}$-Killing spinor, we have
\eq{
    \nabla_{e_j}\varphi = \pm\frac{1}{2}e_j\cdot\varphi, \qquad \D\varphi = \mp\frac{n}{2}\varphi.
}
Moreover, $(M,g)$ is Einstein and
\eq{
    {\rm Ric}=(n-1)g,\qquad R=n(n-1).
}
Using the Schr\"odinger-Lichnerowicz formula \eqref{pointwise_S-L} we have
\eq{
    -\Delta\<\varphi,\psi\> &= \<-\Delta\varphi,\psi\> + \<\varphi,-\Delta\psi\> - 2\<\pm\frac{1}{2}e_j\cdot\varphi,\nabla_{e_j}\psi\>\\
    &= \<\D^2\varphi-\frac{R}{4}\varphi,\psi\> + \<\varphi,\D^2\psi-\frac{R}{4}\psi\> \pm \<\varphi,\D\psi\>\\
    &= \Big( \frac{n^2}{4} - \frac{n^2-n}{4} + \lambda^2 - \frac{n^2-n}{4} \pm\lambda \Big)\<\varphi,\psi\>\\
    &= \Big(\lambda^2\pm\lambda-\frac{n^2-2n}{4}\Big)\<\varphi,\psi\>.
}
    
\end{proof}

\begin{lemma}\label{lem4.2}
    Let $G$ be the endomorphism given in Section \ref{sec2.3}. If $\varphi$ is a $\pm\frac{1}{2}$-Killing spinor, then $G\varphi$ is a $\mp\frac{1}{2}$-Killing spinor.
\end{lemma}

\begin{proof}
Using \eqref{def_G} we have
\eq{
    \nabla_X(G\varphi) = G(\nabla_X\varphi) = G(\pm\frac{1}{2}X\cdot\varphi) = \mp\frac{1}{2}X\cdot G\varphi, \quad\forall X\in\Gamma(TM).
}
Hence we complete the proof.
\end{proof}

Now we prove the main result, Theorem \ref{main_thm}.

\begin{proof}[Proof of Theorem \ref{main_thm} for $n\geq3$]
By Theorem \ref{prop_gamma_a} and Proposition \ref{prop3.5} we have
\eq{
    \lambda^2 \geq  \frac{n}{4(n-1)} \gamma_1([g,f]) \geq \frac{n}{4(n-1)} Y(M,\p M,[g]).
}
Now we prove the equality case.

$(1) \iff (2)$: Assume equality holds, i.e.
\eq{
    \lambda^2 = \frac{n}{4(n-1)} \gamma_1([g,f]) = \frac{n}{4(n-1)} Y(M,\p M,[g]).
} 
Recall \eqref{constant_curvature} that for some $\tilde{g}\in[g]$ we have
\eq{
    R_{\tilde{g}} = Y(M,\p M,[g])f_{\tilde{g}}^2 = \frac{4(n-1)}{n}\lambda^2f_{\tilde{g}}^2, \qquad H_{\tilde{g}}=0.
}
Thus \eqref{crucial_integral_identity} implies $P\varphi_{\tilde{g}}=0$, i.e., $\varphi_{\tilde{g}}$ is a twistor spinor, hence is smooth. We recall the identity for any twistor spinor $\psi$ (see \cite{Baum_Book_90})
\eq{\label{why_varphi_constant_lenth}
    \Delta \abs{\psi}^2 = \frac{R}{2(n-1)}\abs{\psi}^2 - \frac{2}{n}\abs{\D\psi}^2, \qquad \nabla_{\nu}\abs{\psi}^2 = -\frac{2}{n}\rRe\<\psi, \nu\cdot\D\psi\>.
}
Applying to $\psi=\varphi_{\tilde{g}}$  and to $\tilde g$ we see that (we omit the subscript $\tilde{g}$ for simplicity)
\eq{\label{eq1_case_Neumann}
    \Delta \abs{\varphi}^2 = \frac{2}{n}\lambda^2 f^2\abs{\varphi}^2 - \frac{2}{n}\lambda^2 f^2\abs{\varphi}^2 = 0,
}
and 
\eq{\label{eq2_case_Neumann}
    \nabla_{\nu}\abs{\varphi}^2 = -\frac{2}{n}\lambda f \rRe\<\varphi,\nu\cdot\varphi\> = 0.
}
Now \eqref{eq1_case_Neumann} and \eqref{eq2_case_Neumann} implies $\abs{\varphi}\equiv{\rm const}$.
By a well known result (see for instance \cite{Ginoux09}*{Proposition A.2.1})
it implies that $(M,\tilde{g})$ is Einstein. Therefore $R\equiv{\rm const}$, hence $f\equiv{\rm const}=\pm\rVol^{1/n}$. Now \eqref{0-form_eq} and the fact that $\varphi$ is a twistor spinor yield
\eq{
    \nabla_X\varphi = -\frac{1}{n}X\cdot\D\varphi = -\frac{1}{n}\lambda f X\cdot\varphi, \quad\forall X\in\Gamma(TM).
}
Therefore $\varphi=\varphi_{\tilde{g}}$ is a real Killing spinor.
     
For the reverse, if $\varphi$ is a $b$-Killing spinor with $b\not=0$, without loss of generality we may assume $b=\pm\frac{1}{2}$, then $R\equiv n(n-1)$, and an Obata-type uniqueness result for manifolds with boundary (see \cites{Escobar90, Akutagawa21}) implies that
\eq{
    Y(M,\p M, [g]) = R\cdot\rVol^{\frac{2}{n}} = n(n-1)\rVol^{\frac{2}{n}}.
}
and
\eq{
    \D\varphi = -nb\varphi = \mp\sqrt{\frac{nR}{4(n-1)}}\varphi = \mp\sqrt{\frac{n}{4(n-1)}Y(M,\p M,[g])}\cdot\rVol^{-\frac{1}{n}}\varphi \eqcolon \lambda f\varphi.
}
Hence the equality in \eqref{ineq_thm1.1} holds.

$(1) \iff (3)$: Assume equality holds. (1) shows that (up to a conformal transformation) $\varphi$ is a $\pm\frac{1}{2}$-Killing spinor. Hence (up to a conformal transformation) $(M,g)$ is Einstein and
\eq{
    {\rm Ric}=(n-1)g,\qquad R=n(n-1).
}
Moreover, Lemma \ref{lem4.2} implies that $G\varphi$ is a $\mp\frac{1}{2}$-Killing spinor, hence a $\pm\frac{n}{2}$-eigenspinor of $\D$. Thus Lemma \ref{lem4.1} implies
\eq{\label{eq1_Reilly_Dirichlet}
    -\Delta\<\varphi,G\varphi\> = n\<\varphi,G\varphi\>.
}
By \eqref{def_G} we see that $G$ is symmetric, hence $\<\varphi,G\varphi\>\in\mathbb{R}$. Using the boundary condition $\mathbf{B}\varphi = 0$ and \eqref{vanishing_boundary_nu_cdot} we see that
\eq{\label{eq2_Reilly_Dirichlet}
    \<\varphi,G\varphi\> = 0 \quad\hbox{on}\ \p M.
}
Moreover, applying \eqref{general_not_self_adjoint} to $\psi=G\varphi$ yields
\eq{
    n\int_M \<\varphi, G\varphi\> = \mp\int_{\p M}\<\nu\cdot G\varphi,\varphi\>.
}
The boundary condition $\mathbf{B}\varphi = 0$ ensures that the right-hand side is non-zero, hence
\eq{\label{eq3_Reilly_Dirichlet}
    \<\varphi,G\varphi\> \not\equiv 0.
}
Since ${\rm Ric} = (n-1)g$, a Lichnerowicz-Obata type result for the boundary case proved by Reilly \cite{Reilly77} implies that the first Dirichlet-eigenvalue of $-\Delta$ satisfies $\lambda_1(-\Delta)\geq n$ with equality if and only if $(M,g)$ is isometric to the hemisphere $(\S^n_+,g_{{\rm st}})$. Now \eqref{eq1_Reilly_Dirichlet}, \eqref{eq2_Reilly_Dirichlet} and \eqref{eq3_Reilly_Dirichlet} show that $\<\varphi,G\varphi\>$ is a non-zero $n$-Dirichlet-eigenfunction of $-\Delta$. Hence the conclusion follows.

For the reverse, assume that $(M,g)$ is conformally equivalent to the hemisphere $(\S^n_+,g_{{\rm st}})$. Since \eqref{0-form_eq} is conformally invariant, we may assume $(M,g)=(\S^n_+,g_{{\rm st}})$ and admits a non-trivial solution $(\varphi,f)$ to \eqref{0-form_eq} with $\varphi$ a $\pm\frac{1}{2}$-Killing spinor. We need to show that it achieves equality in \eqref{ineq_thm1.1}.
In fact, in this case we have $\D\varphi = \mp\frac{n}{2}\varphi$ and hence $\lambda f=\mp\frac{n}{2}$.
Using $\norm{f}_n=1$ and \eqref{relative_Yamabe_constant_hemisphere} we have
\eq{
    \lambda^2 = \norm{\lambda f}_n^2 = \norm{\frac{n}{2}}_n^2 = \frac{n^2}{4}\Big(\frac{\omega_n}{2}\Big)^{\frac{2}{n}} = \frac{n}{4(n-1)}Y(\S^n_+,\p\S^n_+).
}
Hence we complete the proof.

\end{proof}

\begin{remark}\label{rmk_generalization}
The proof shows how inequality \eqref{ineq_thm1.1} follows from \eqref{crucial_integral_identity}. It works more generally, see Section \ref{sec_generalization}.
In fact, one can similarly prove that inequality \eqref{ineq_thm1.1} holds true for any function $F\in L^n$: if there exists a non-trivial spinor field $\varphi$ such that
\eq{
    \abs{\D \varphi} \leq \abs{F}\abs{\varphi}\quad\hbox{in}\ M, \qquad \mathbf{B}\varphi = 0\quad\hbox{on}\ \p M,
}
then
\eq{
    \norm{F}_{L^n}^2 \geq \frac{n}{4(n-1)} Y(M,\p M,[g]).
}

\end{remark}

Now we prove the main result for $n=2$. In this case, every oriented surface is spin, and the relative Yamabe constant is $4\pi\chi(\Sigma)$.

\begin{proof}[Proof of Theorem \ref{main_thm} for $n=2$]
We only need to consider $\chi(\Sigma)>0$, hence $\Sigma$ is a topological disk and $\chi(\Sigma)=1$. 

Since $\varphi\in L^2$ and $A\in L^q$ with $q>2$, we have $A\cdot\varphi\in L^{\frac{2q}{q+2}}$. Using \eqref{0-form_eq} and the apriori estimate \eqref{apriori_estimate} we can see $\varphi\in L^q$. Therefore Lemma \ref{lem_regularity_L^r} implies that $\varphi\in L^r$ for any $r>2$. Using \eqref{apriori_estimate} again we have $\varphi\in W^{1,2}$.

Note that Lemma \ref{conformal_invariance} remains true. In fact, for $\tilde{g}=e^{2h}g$ we have
\eq{
    \D_{\tilde{g}}(e^{-\frac{1}{2}h}\tilde{\varphi}) = e^{-\frac{3}{2}h}\widetilde{\D_g\varphi}
}
and the others are the same. In 2-dimensional case we have the Gau\ss{} curvature and boundary geodesic curvature
\eq{
    K_{\tilde{g}} = e^{-2h}(-\Delta_g h + K_g), \qquad k_{\tilde{g}} = e^{-u}(\frac{\p h}{\p \nu} + k_g).
}
Recall our normalization $\int f^2 = 1$. By the Gau\ss{}-Bonnet theorem it follows
\eq{
    \int_\Sigma ( 2\pi f^2 - K_g ) + \int_{\p \Sigma} -k_g = 0.
}
Hence there exists some $u_0$ solving
\eq{
    \begin{cases}
        -\Delta_g u_0 + K_g = 2\pi f^2 \quad\hbox{in}\ M,\\
        \frac{\p u_0}{\p \nu} = -k_g \quad\hbox{on}\ \p M.
    \end{cases}
}
For $\tilde{g}=e^{2u_0}g$ we have
\eq{\label{dim_2_eq1}
    \frac{1}{2}R_{\tilde{g}} = K_{\tilde{g}} = e^{-2u_0}(-\Delta_g u_0+K_g) = e^{-2u_0}2\pi f^2 = 2\pi f_{\tilde{g}}^2
}
and
\eq{
    k_{\tilde{g}} = e^{-u_0}(\frac{\p u_0}{\p \nu} + k_g) = 0,
}
where we have used the analogy of \eqref{conformal_f}.
Together with \eqref{vanishing_boundary_Dirac} and the integral Schr\"odinger-Lichnerowicz formula \eqref{integral_S-L} we have (from now on we omit the subscript $\tilde{g}$)
\eq{\label{eq_thm4.5}
    0 = \int_\Sigma \Big( \frac{K}{2}\abs{\varphi}^2 - \frac{1}{2}\abs{\D\varphi}^2 \Big) + \int_\Sigma \abs{P\varphi}^2 \geq \int_\Sigma \Big( \pi - \frac{1}{2}\lambda^2 \Big) f^2\abs{\varphi}^2.
}
Hence $\lambda^2\geq 2\pi = 2\pi\chi(\Sigma)$. The equality case is treated similarly as in the case $n\geq3$.
    
\end{proof}


\section{Generalizations}\label{sec_generalization}

\subsection{Generalizations of the Dirac equation}

We view a function $f$ as an endomorphism acting on spinor fields by scalar multiplication. Our approach can be generalized to deal with a general symmetric endomorphism. We consider
\eq{\label{endomorphism_eq}
    \begin{cases}
        \D\varphi = \lambda H\varphi \quad\hbox{in}\quad M,\\
        \mathbf{B}^\pm \varphi = 0 \quad\hbox{on}\quad \partial M,
    \end{cases}
}
where $H\in \Gamma({\rm End}^{{\rm sym}}(\Sigma M))$ is symmetric and normalized as $\int_M \norm{H}^n = 1$. Here we denote by $\norm{H}$ the operator norm of $H$ and by $\norm{H}_n$ the $L^n$-norm of $\norm{H}$. 

For a general $H\in \Gamma({\rm End}^{{\rm sym}}(\Sigma M,g))$, we have the similar conformal invariance.
Let $\tilde{g}=e^{2h}g$. The isometry $(TM,g) \to (TM, \tilde{g})$ with $X\mapsto \tilde{X}=e^{-h}X$ induces a canonical isometry $F: (\Sigma M,g) \to (\Sigma M, \tilde{g})$ with $\varphi \mapsto \tilde{\varphi}$. We define the induced $\tilde{H}$ by
\eq{
    \tilde{H} \coloneqq F \circ H \circ F^{-1} : \Gamma(\Sigma M,\tilde{g})\to \Gamma(\Sigma M,\tilde{g}).
}
By definition we have
\eq{
    \tilde{H}\tilde{\varphi} = \widetilde{H\varphi},
}
hence $\tilde{H}\in\Gamma({\rm End}^{{\rm sym}}(\Sigma M,\tilde{g}))$ and $\norm{\tilde{H}}_{\tilde{g}}=\norm{H}_g$.

\begin{lemma}[Conformal invariance]
Equation \eqref{endomorphism_eq} is conformally invariant in the following sense: Let $(\varphi,H)$ be a solution to \eqref{endomorphism_eq} w.r.t. metric $g$, and $\tilde g = e^{2h} g$.
Define 
\eq{\label{conformal_H}
     \varphi_{\tilde{g}} \coloneqq e^{-\frac{n-1}{2}h}\tilde{\varphi}, \qquad H_{\tilde{g}} \coloneqq e^{-h}\tilde{H}.
}
Then $(\varphi_{\tilde g},H_{\tilde g})$ is a solution to \eqref{0-form_eq} w.r.t. metric $\tilde{g}$ with the same $\lambda$.
\end{lemma} 

\begin{proof}
In fact, \eqref{conformal_Dirac} and \eqref{endomorphism_eq} imply
\eq{
    \D_{\tilde{g}}\varphi_{\tilde{g}} = e^{-\frac{n+1}{2}h}\widetilde{\D_g\varphi} = \lambda\, e^{-\frac{n+1}{2}h}\tilde{H}\tilde{\varphi} = \lambda\, e^{-\frac{n+1}{2}h}\cdot e^{h}H_{\tilde{g}}\cdot e^{\frac{n-1}{2}h}\varphi_{\tilde{g}} = \lambda H_{\tilde{g}}\varphi_{\tilde{g}},
}
where
\eq{
    \int \norm{H_{\tilde{g}}}_{\tilde{g}}^n  \rd V_{\tilde{g}} = \int \norm{e^{-h}H}_g^n \cdot e^{nh}\, \rd V_g = \int \norm{H}_g^n \rd V_g = 1.
}
Moreover, since $\mathbf{B}\varphi=0$ on $\p M$, we have
\eq{
    \tilde{\mathbf{B}}\varphi_{\tilde{g}} = \frac{1}{2}(\rid \pm \tilde{\nu}\,\tilde{\cdot}\, \tilde{G})\varphi_{\tilde{g}} = e^{-\frac{n-1}{2}h} \frac{1}{2} ( \tilde{\varphi} \pm \widetilde{ \nu\cdot G\varphi } ) = e^{-\frac{n-1}{2}h} \widetilde{\mathbf{B}\varphi} = 0 \quad\hbox{on}\ \p M.
}

\end{proof}

Note that by the definition of operator norm we have
\eq{\label{def_operator_norm}
    \norm{H} = \sup_{\varphi\not\equiv0} \frac{\abs{H\varphi}}{\abs{\varphi}}.
}
Hence for any spinor field $\varphi$ we have
\eq{
    \abs{H\varphi} \leq \norm{H}\cdot\abs{\varphi}.
}
In view of Remark \ref{rmk_generalization}, using the same argument as in Section \ref{sec_a_family_of_conformal_invariants} we can prove the following theorem.

\begin{theorem}\label{generalization_thm}
Let $(M^n,g)$ $(n\geq2)$ be a compact spin manifold with boundary $\p M$. If $(M,g)$ admits a non-trivial solution $(\varphi,H)$ to \eqref{endomorphism_eq}, where $\varphi\in L^p$ with $p>\frac{n}{n-1}$ and $\norm{H}\in L^n$ with $\norm{H}_n = 1$, then
\eq{\label{ineq_thm5.2}
    \lambda^2 \geq \frac {n}{4(n-1)} Y(M,\p M,[g]).
}
Equality holds if and only if (up to a conformal transformation) $\varphi = \varphi_+ + \varphi_-$, where $\varphi_\pm$ is a $\mp \frac{1}{2}$-Killing spinor with $\rRe\<\varphi_+,\varphi_-\>=0$, and $\varphi$ achieves the supreme in \eqref{def_operator_norm}.

\end{theorem}

\begin{proof}
The proof is the same as that of \cite{WZ25_zero_mode}*{Theorem 5.10}. The only difference is, instead of the classical Obata's rigidity, we need an analogical result for manifolds with boundary (see \cites{Escobar90, Akutagawa21}).

\end{proof}

\begin{remark}
Here for the rigidity, it is difficult to go further. In fact, for $\varphi = \varphi_+ + \varphi_-$, if we normalize $b=-\frac{1}{2}$ and ${\rm Ric}=(n-1)g$ as above, then for
\eq{
   \bar{\varphi} \coloneqq \varphi_+ - \varphi_-,
}
one can easily check that  $G\bar{\varphi}=G\varphi_+-G\varphi_-$ is also a twistor spinor and
\eq{
    \D\varphi = \frac{n}{2}\varphi, \quad \D\bar{\varphi} = \frac{n}{2}\bar{\varphi}, \quad \D G\varphi = -\frac{n}{2}G\bar{\varphi}, \quad \D G\bar{\varphi} = -\frac{n}{2}G\varphi.
}
Therefore
\eq{
    -\Delta \<G\varphi,\bar{\varphi}\> = n\<G\varphi,\bar{\varphi}\> \quad\hbox{with}\quad \<G\varphi,\bar{\varphi}\> \not\equiv0.
}
However, the function $\<G\varphi,\bar{\varphi}\>$ does not satisfy a suitable boundary condition, which prevents us to apply any known version of Obata-type rigidity.    
\end{remark}

\textbf{Example 1:} If $H = iA\cdot$, where $A$ is a vector field and $\cdot$ is the Clifford multiplication, then
\eq{
    \norm{H} = \abs{A}.
}
The corresponding equation is called the zero mode equation, see for instance \cites{Loss_Yau_86, FL1}. For the case of a closed manifold, the sharp lower bound of $\lambda^2$ was proved by \cites{Frank_Loss_2024, Reuss25, WZ25_zero_mode}.

\textbf{Example 2:} If $H = i\alpha\cdot$, where $\alpha$ is a differential $2$-form, then
\eq{
    \norm{H} = \big[\frac{n}{2}\big]^{\frac{1}{2}}\cdot\abs{\alpha}.
}
The corresponding equation was studied in \cite{WZ25_zero_mode}, where the similar sharp lower bound for a closed manifold was proven.

\subsection{Generalizations of the boundary condition}

The same inequalities hold true for another boundary condition, the MIT bag boundary condition, defined by
\eq{
    \mathbf{B}_{{\rm MIT}}^\pm\varphi \coloneqq \frac{1}{2}(\rid \pm i\nu\cdot)\varphi = 0.
}
More generally, we consider the $J$-boundary condition defined by (see for instance \cite{CJSZ2018})
\eq{
    \mathbf{B}_{J}^\pm\varphi \coloneqq \frac{1}{2}(\rid \pm \nu\cdot J)\varphi = 0,
}
where $J\in \Gamma({\rm End}^{{\rm anti}}(\Sigma M))$ satisfies
\eq{
    J^2 = -\rid, \quad J^* = -J, \quad \nabla J = 0, \quad X\cdot J\varphi = J(X\cdot\varphi), \quad\forall X\in\Gamma(TM).
}
In particular, $J=i$ is a canonical choice which gives the MIT bag.
Another choice is $J=iG_1G_2$, where $G_1,G_2$ are chiral endomorphisms given in \eqref{def_G} with $[G_1,G_2]=0$.

The MIT bag and $J$-boundary conditions are also local elliptic, and one can easily check all the crucial properties, \eqref{vanishing_boundary_Dirac} and Lemma \ref{conformal_invariance} remain true. 
The only difference in this case is that \eqref{vanishing_boundary_nu_cdot} is not true. In fact, it is easy to check that $f$ or $H$ here need to be complex, instead of real.
Therefore we similarly have the inequalities in Theorem \ref{main_thm} and Theorem \ref{generalization_thm}.


\section{An application}

In this section we give an interesting application of Theorem \ref{main_thm}.

On a compact spin manifold with boundary, we consider the energy functional
\eq{
    \mathcal{L}(\varphi) \coloneqq \frac{1}{2}\int_M \<\D\varphi,\varphi\> \rdV_g - \frac{n-1}{2n} \int_M \abs{\varphi}^{\frac{2n}{n-1}} \rdV_g + \frac{1}{2}\int_{\p M}\< \nu\cdot\mathbf{B}\varphi,\varphi\> \rd s_g.
}
The choice of the boundary term comes from the following fact:
\eq{
    (\psi,\varphi) \mapsto \int_M \<\D\psi,\varphi\> \rdV_g + \int_{\p M}\< \nu\cdot\mathbf{B}\psi,\varphi\> \rd s_g
}
is a symmetric bilinear form. This can be easily checked by using \eqref{general_not_self_adjoint}. Consequently, $J(\varphi)$ is a real number. Given any variation $\phi \coloneqq \frac{\rd\varphi_t}{\rd t}|_{t=0}$ with $\varphi_0=\varphi$, the first variation of $\mathcal{L}$ is
\eq{
    \frac{\rd }{\rd t}\Big|_{t=0}  \mathcal{L}(\varphi_t) = \int_M \rRe\< \D\varphi - \abs{\varphi}^{\frac{2}{n-1}}\varphi, \phi \> + \frac{1}{2}\int_M \rRe\< \nu\cdot\mathbf{B}\varphi,\phi\>.
}
Therefore, the corresponding Euler-Lagrange equation is
\eq{
    \begin{cases}\label{E-L-eq2}
        \D\varphi = \abs{\varphi}^{\frac{2}{n-1}}\varphi\quad\hbox{in}\ M, \\
        \mathbf{B}\varphi = 0 \quad\hbox{on}\ \p M.
    \end{cases}
}

As an application of Theorem \ref{main_thm}, we immediately have the following energy gap.

\begin{theorem}
For any solution $\varphi$ to \eqref{E-L-eq2} we have
\eq{
    \mathcal{L}(\varphi) \geq \frac{1}{2n}\left(\frac{n}{4(n-1)}Y(M,\p M,[g])\right)^{\frac{n}{2}}. 
}
Equality holds if and only if (up to a conformal transformation) $(M,g)=(\S^n_+,g_{{\rm st}})$ and $\varphi$ is a $-\frac{1}{2}$-Killing spinor.
\end{theorem}

\begin{proof}
Let $\varphi$ be a solution to \eqref{E-L-eq2}. Theorem \ref{main_thm} implies
\eq{
    \Big(\int_M \abs{\varphi}^{\frac{2n}{n-1}}\Big)^{\frac{2}{n}} = \norm{\varphi}_n^2 \geq \frac{n}{4(n-1)}Y(M,\p M,[g]).
}
Using \eqref{E-L-eq2} we have
\eq{
    \mathcal{L}(\varphi) = \frac{1}{2n}\int_M \abs{\varphi}^{\frac{2n}{n-1}}.
}
Hence we obtain the desired inequality. The equality case follows from Theorem \ref{main_thm} as well.

\end{proof}

For $(M,g)=(\S^n,g_{{\rm st}})$ (without boundary), the analogous inequality was proved by \cite{Isobe11} and the equality case was classified by \cite{Borrelli21}.

\appendix

\section{The regularity}

\begin{lemma}\label{lem_regularity_L^r}
    If $\varphi\in L^p$ is a solution to \eqref{0-form_eq} with $p>\frac{n}{n-1}$, then $\varphi\in L^r$ for any $\frac{n}{n-1}<r<\infty$.
\end{lemma}

\begin{proof} 

The proof follows completely the one in \cite{FL1} with one exception that we use the ellipticity theory instead of the Hardy-Littlewood-Sobolev (HLS) inequality.

The argument is the same as \cite{WZ25_zero_mode}*{Lemma A.1}. The only difference is that we need to use the following elliptic apriori estimate for the Dirac operator with boundary condition (see for instance \cite{CJSZ2018}): let $\psi$ be a solution to
\eq{
    \begin{cases}
        \D\psi = \chi\quad\hbox{in}\ M,\\
        \mathbf{B}\psi = 0\quad\hbox{on}\ \p M,
    \end{cases}
}
and $\chi\in L^q$ with $1<q<\infty$, then
\eq{\label{apriori_estimate}
    C^{-1}\norm{\psi}_{\frac{nq}{n-q}} \leq \norm{\psi}_{W^{1,q}}\leq C \norm{\chi}_q
}
for some $C=C(n,q)>0$.

\end{proof}

\noindent{\sc Acknowledgment.}
The author would like to sincerely thank Prof. Guofang Wang for his constant encouragement and insightful discussions.
The author also appreciate the financial support from CSC (No. 202306270117).

\bibliographystyle{alpha}
\bibliography{BibTemplate.bib}

\end{document}